\documentclass[11.5pt]{amsart}
\makeatletter
\@fleqnfalse
\makeatother
\usepackage{etex}
\usepackage{amsmath, amssymb}
\usepackage[frame,cmtip,arrow,matrix,line,graph,curve]{xy}
\usepackage{graphpap, color, paralist, pstricks}
\usepackage[mathscr]{eucal}
\usepackage[pdftex]{graphicx}
\usepackage[pdftex,colorlinks,backref=page,citecolor=blue]{hyperref}
\usepackage{tikz}
\usepackage{kotex}
\usepackage{array}
\newcolumntype{P}[1]{>{\centering\arraybackslash}p{#1}}
\newcolumntype{M}[1]{>{\centering\arraybackslash}m{#1}}
\usepackage{caption}
\usepackage{multirow}

\newtheorem{theorem}{Theorem}[section]
\newtheorem{proposition}[theorem]{Proposition}
\newtheorem{corollary}[theorem]{Corollary}
\newtheorem{lemma}[theorem]{Lemma}

\newtheorem{conjecture}[theorem]{Conjecture}

\theoremstyle{definition}

\newtheorem{remark}[theorem]{Remark}

\newcommand{\PP}{\mathbb{P}}

\newcommand{\CC}{\mathbb{C}}

\newcommand{\ZZ}{\mathbb{Z}}

\newcommand{\cO}{\mathcal{O} }

\newcommand{\cF}{\mathcal{F} }

\newcommand{\cI}{\mathcal{I} }
\newcommand{\cK}{\mathcal{K} }
\newcommand{\cM}{\mathcal{M} }
\newcommand{\cN}{\mathcal{N} }

\newcommand{\cS}{\mathcal{S} }

\newcommand{\rH}{\mathrm{H} }

\newcommand\bG{\mathbf{G}}
\newcommand\bH{\mathbf{H}}

\newcommand\bM{\mathbf{M}}

\newcommand{\Coh}{\text{Coh}}
\newcommand{\DT}{\mathrm{DT}}

\newcommand\Hom{\mathrm{Hom} }
\newcommand\Ext{\mathrm{Ext} }

\newcommand{\Gr}{\mathrm{Gr} }

\newcommand\SL{\mathrm{SL}}

\newcommand\lr{\rightarrow}

\newcommand{\ses}[3]{0\lr{#1}\lr{#2}\lr{#3}\lr 0}

\hypersetup{%
citecolor=blue,%
linkcolor=blue,%
}

\begin{document}

\title[DT-GV correspondences on the Mukai-Umemura variety]{DT-GV correspondence on the Mukai-Umemura variety}

\author{Kiryong Chung}
\address{Department of Mathematics Education, Kyungpook National University, 80 Daehakro, Bukgu, Daegu 41566, Korea}
\email{krchung@knu.ac.kr}

\author{Joonyeong Won}
\address{Department of Mathematics, Ewha Womans University, 52, Ewhayeodae-gil, Seodaemun-gu, Seoul, 03760, Republic of Korea} 
\email{leonwon@ewha.ac.kr}


\begin{abstract}
We compute Donaldson--Thomas (DT) invariants and their descendant invariants 
for the local Calabi--Yau $4$--fold over the Mukai--Umemura variety via several localization formulas. Assuming that the genus-one Gopakumar--Vafa (GV) type invariants vanish, our computations verify the predictions of Cao, Maulik, and Toda.
\end{abstract}

\maketitle
\section{Introduction}
\subsection{Motivation and result}
Let $Y$ be a Calabi--Yau $4$-fold and 
$\beta \in H_2(Y,\mathbb{Z})$ be a curve class. 
Let $\bM_\beta(Y)(=\bM)$ denote the moduli space of one-dimensional 
stable sheaves $F$ on $Y$ satisfying $[F]=\beta$ and $\chi(F)=1$. 
Given a class $\gamma \in \rH^{4-2i}(Y,\mathbb{Z})$, 
the \emph{insertion} is defined by
\[
\tau_i(\gamma)
=
\pi_{\bM*}\!\left(
\pi_Y^*\gamma \cup \mathrm{ch}_{i+3}(\mathcal{F})
\right),
\]
where $\cF\in \mathrm{Coh}(\bM_\beta(Y)\times Y)$ is the universal sheaf with the \emph{trivial determinant} (i.e., $\det(R\pi_{\bM*}\cF)=\cO_{\bM}$). The primary (for $i=0$) or descendent (for $i\geq1$) invariants are defined by the integration of the insertion over the virtual class $[\bM_{\beta}(Y)]^{\mathrm{vir}}$:
\begin{equation}\label{intinse}
\langle \tau_i(\gamma) \rangle_{\beta} := \int_{[\bM_\beta(Y)]^{\mathrm{vir}}} \tau_i (\gamma).
\end{equation}
The integral $\langle \tau_0(\gamma) \rangle_{\beta}$ is called \emph{$\DT$-invariants} (or \emph{primary invariant}).
When $i \geq 1$, we call $\tau_i(\gamma)$ \emph{descendant insertions}, and its integrations over the virtual class are called \emph{descendant invariants}.

The DT theory of Calabi--Yau 4-folds has recently attracted significant attention in enumerative geometry, motivated by its rich geometric structures and deep connections to other curve counting theories. In particular, Cao, Maulik, and Toda proposed a series of conjectures predicting a precise correspondence between DT invariants of Calabi--Yau 4--folds and the Gopakumar--Vafa (GV) type invariants introduced by Klemm and Pandharipande (\cite{CMT18, CT21, KP08}). More precisely, it states that
\begin{conjecture}[\protect{\cite[Conjecture 0.2]{CMT18} and \cite[Conjecture 0.2]{CT21}}]\label{conjor}
Let $n_{0,\beta}(-)$ (resp. $n_{1,\beta}$) denote genus $0$ (resp. $1$) GV-type invariants and $m_{\beta_1,\beta_2}$ be the meeting invariants defined in \cite{KP08}. Under the suitable orientation on the moduli space $\bM_{\beta}(Y)$, we have
\begin{enumerate}[(a)]
\item $ n_{0,\beta}(\gamma_0)=\langle \tau_0(\gamma_0) \rangle_{\beta}$ for $\gamma_0\in \rH^{4}(Y, \ZZ)$.
\item Furthermore, we have an linear identity:
\begin{equation}\label{eq:mov1}
\langle \tau_1(\gamma_1) \rangle_\beta
=
\frac{n_{0,\beta}(\gamma_1^2)}{2(\gamma_1\cdot\beta)}
-
\sum_{\beta_1+\beta_2=\beta}
\frac{(\gamma_1\cdot\beta_1)(\gamma_1\cdot\beta_2)}
{4(\gamma_1\cdot\beta)}
\, m_{\beta_1,\beta_2}
-
\sum_{k\ge1,\,k|\beta}
\frac{\gamma_1\cdot\beta}{k}\,
n_{1,\beta/k}
\end{equation}
with respect to $\gamma_1\in \rH^2(Y,\ZZ)$.
\end{enumerate}
\end{conjecture}
In this paper, we study the case where $Y = \mathrm{Tot}(K_X)$ is the total space of the canonical bundle of the \emph{Mukai-Umemura variety} $X$ (i.e. the unique smooth prime Fano threefold of genus $12$ admitting a nontrivial $\mathrm{SL}_2(\mathbb{C})$-action). For the defining relations of the meeting invariants $m_{\beta_1,\beta_2}$ for our local Fano $4$-fold $Y$, see the first paragraph of Section \ref{subproof}.

It was previously known that Conjecture \ref{conjor} holds for curve classes $\beta = d[\mathrm{line}]$, $1 \le d \le 3$ and $Y= \mathrm{Tot}(K_X)$ (\cite[Theorem 4.2]{CLW24}). For $d \le 3$, the explicit descriptions of the universal sheaf on the moduli space obtained in \cite[Lemma 3.1]{AF06}, \cite[Lemma 4.1]{Fae14}, and \cite[Theorem 2.4]{KS04} are sufficient to prove the conjecture. However, the case $d=4$ requires a more delicate analysis of the structure of stable sheaves and their contributions to the virtual cycle. In this paper, we extends our results degree up to $d=4$.
\begin{theorem}[$=$Theorem \ref{pfmeeting1}]\label{maino}
For $Y=\mathrm{Tot}(K_X)$ with the Mukai-Umemura variety $X$, the equality \eqref{eq:mov1} in Conjecture \ref{conjor} holds for $\beta = 4[\mathrm{line}]$ whenever we assume that $n_{1, d}=0$ for $1\leq d\leq 4$.
\end{theorem}
The proof of Theorem \ref{maino} is based on a detailed analysis of the torus-fixed stable sheaves in the moduli space (Proposition \ref{multfil} and Proposition \ref{pr:l2q}). We then apply various (virtual) localization formulas studied in \cite{AB84, BV82, GP99, Tho92, Tho93}. By organizing these contributions via localization, we explicitly compute the primary and descendent invariants (Proposition \ref{mainpro}). This establishes the equality \eqref{eq:mov1} in Conjecture \ref{conjor} (Section \ref{subproof}). In the computations, one of the key features is that for any fixed curve other than  the multiplicity $4$-lines, there always exists a weight-zero component in the obstruction space (Lemma \ref{d1l}, Proposition \ref{pr:l2q} and Proposition \ref{qll}), which prevents it from contributing to the localization formula. Moreover, the same phenomenon occurs for $d\leq 3$ and thus we can easily recover the our previous results (\cite[Proposition 3.17]{CLW24}).
\subsection*{Acknowledgements}
The authors gratefully acknowledge the many helpful suggestions and comments of Sanghyeon Lee during the preparation of the paper. The first named author was supported by the National Research Foundation of Korea NRF2021R1I1A3045360. The second named author was supported by the National Research Foundation of Korea RS-2025-00513064.

\section{Preliminary}
In this section, we present several localization formulas and recall results from \cite{CKK25} that will be used later.
\subsection{Localization formulas for varieties with isolated fixed points}
Let $T$ ($\cong \CC^*$) be an algebraic torus and $M$ be a $T$-space whose fixed locus $M^T $ is finite. We work in equivariant cohomology $H_T^*(-)$ and homology $\rH_*^T(-)$. Let us denote $e^T(-)$ by the equivariant Euler class. 
\begin{theorem}[Virtual localization \protect{\cite{GP99}}]\label{virloc}
Assume that $M$ carries a $T$-equivariant perfect obstruction theory,
with virtual fundamental class $[M]^{\mathrm{vir}} \in H^T_*(M)$.
Then for any $u \in H_T^*(M)$,
\begin{equation}\label{eq:virtual-localization}
\int_{[M]^{\mathrm{vir}}} u
=
\sum_{x \in M^T}
\frac{u|_x}
     {e^T\!\bigl(T_x^{\mathrm{vir}} M\bigr)}.
\end{equation}
\end{theorem}

\begin{remark}[Smooth case \protect{\cite{AB84, BV82}}]\label{smvir}
If $M$ is smooth, for any $u \in H_T^*(M)$, we have
\begin{equation}
\int_{[M]} u
=
\sum_{x \in M^T}
\frac{u|_x}{e^T(T_xM)}.
\end{equation}
\end{remark}
\subsection{Some facts in equivariant K-theory}
In \cite{Tho92, Tho93}, the author proves an equivariant self-intersection formula and Lefschetz-type localization formula in algebraic K-theory for torus actions. For later use, we collect some results. Let us denote $$\lambda_{-1}(E|_p)=\sum_{i=0}^{\mathrm{rank}\, E}(-1)^i\,[\wedge^i E|_p]$$ by the \emph{$K$-theoretic $\lambda$--class} of a $T$-equivariant locally free sheaf $E$ on $X$ at a fixed point $p\in X$. Denote by $R(T)$ the representation ring of $T$, which is canonically isomorphic to $R(\CC^*)\cong \ZZ[t^{\pm1}]$, where $t$ denotes the standard one-dimensional character of $\CC^*$. We define the \emph{$\CC^*$-equivariant Euler characteristic} (or \emph{pairing}) by
\[
\chi(E,F):=\sum_{i}(-1)^i [\Ext^i(E,F)]\in R(\CC^*)\]
for $E$ and $F$ are $T$-equivariant coherent sheaves on $X$. 
\begin{theorem}[Theorem 3.5 \protect{\cite{Tho92}}]\label{eqeuler}
Let $X$ be a smooth projective variety with a torus action $T$ and let
$F,G \in \Coh^T(X)$.
Assume that the fixed locus $X^T$ consists of isolated points.
Then for $i_p:\{p\}\hookrightarrow X$, we have
\[
\chi(F,G)
=
\sum_{p \in X^T}
\frac{i_p^*([F])^\vee \cdot i_p^*([G])}
{\lambda_{-1}(T_p^*X)}.
\]
\end{theorem}
We recall the \emph{(equivariant) self-intersection formula} in the $K$-theory.
\begin{theorem}[\protect{\cite[Theorem 3.1]{Tho93}}]\label{self}\label{selfin}
Let $i\colon C\hookrightarrow X$ be an embedding map such that $C$ is a locally complete intersection.
Let $p\in C^{T}$ be a $T$-fixed point. Then for any $F\in \Coh^{T}(C)$,
\[
i_p^*\bigl([i_*F]\bigr)= [F]_p\cdot \lambda_{-1}\!\bigl(N^*_{C/X, p}\bigr)
\in K_T(\mathrm{pt}).
\]
Here $[F]_p= [F \otimes^{\mathbf L}_{\cO_C} \CC_p] \in K_T(\mathrm{pt})$ denotes the derived fiber of $F$ at the point $p$.
\end{theorem}
\begin{remark}
In Theorem \ref{selfin}, if $F$ is locally free on $C$ around the closed point $p$, then $[F]_p$ coincides with the class
of the $T$-representation on the usual fiber $F|_p$.
\end{remark}
\begin{corollary}\label{freechi}
\label{lflfconc}
Suppose that $X$ is smooth and that the fixed point locus $X^{T}$ is isolated.
Then for any $\mathbb C^{*}$-equivariant locally free sheaf $E$ on $X$, we have
\[
\chi(E)
=
\sum_{p\in X^{T}}
\frac{[E]_p}{\lambda_{-1}(T_p^*X)}.
\]
\end{corollary}
\begin{proof}
Let $F=\cO_X$ and $G=E$. Then by Theorem \ref{eqeuler} and Theorem \ref{selfin}, we get the result.
\end{proof}
\begin{corollary}
\label{smcurvechi}
Let $X$ be a smooth projective variety with a $T$-action, and let
$i \colon C \hookrightarrow X$ be a smooth $T$-invariant curve.
Then
\begin{equation}
\chi(i_*\cO_C, i_*\cO_C)
=\sum_{i=0}^{\dim X-1}(-1)^i\chi(C,\, \wedge^i N_{C/X}),
\end{equation}
where $\wedge^0 N_{C/X}=\cO_C$.
\end{corollary}
\begin{proof}
Applying $F=G=i_*\cO_C$ in Theorem \ref{eqeuler}, a direct calculation provides the result by Theorem \ref{selfin} and Corollary \ref{freechi}.
\end{proof}
\subsection{Geometric data of the variety $X$}\label{subge}
Let us recall the definition of the Mukai-Umemura variety $X$. Let $\cS$ be the universal bundles of Grassmannian $\Gr(3,7)$. Then $X$ is defined as a zero locus of a regular section of $(\wedge^2\cS^*)^{\oplus3}$. Furthermore, 
\begin{itemize}
\item The cohomology ring of $X$ over $\ZZ$ is isomorphic to
\[
\rH^{*}(X,\ZZ)\cong\ZZ[h_1,h_2,h_3]/\langle h_1^2-22h_2,h_1^3-22h_3,h_1h_2-h_3,h_2^2\rangle
\]
where $\deg(h_i)=2i, 1\leq i\leq 3$. Moreover, $h_1=c_1(\cO_{X}(1))$ and $h_i$ is the Poincar\'e dual of the linear space of dimension $3-i$ for $i=2,3$.
\item The Chern classes of the tangent bundle of $X$ is  $c(T_{X})=1+h_1+24h_2+4h_3$.
\item $\mathrm{deg}(X)=22$, $\mathrm{Pic}_{\ZZ}(X)\cong \ZZ\langle H\rangle$, and $K_{X}=-H$ for the hyperplane divisor $H$ of $X$.
\end{itemize}
\subsection{Local charts of the variety $X$ and its weights}
At the beginning part of Section $4$ in\cite{CKK25}, the local charts of the variety $X$ were studied under the existence of an $\SL_2:=\SL(2, \CC)$-action. We now summarize the results that will be used in the sequel. Since we will consider the action of the maximal torus subgroup ($\{\mathrm{diag}(t^{-1}, t)\mid t\in \CC^*\}\cong \CC^*$) of $\SL_2$, let us denote the vector space $U_6=\langle e_6, e_4, e_2, e_0, e_{-2}, e_{-4}, e_{-6}\rangle$ with weights $e_k=-k$ for $k\in \{\pm6, \pm 4, \pm 2, 0\}$. A computation of torus weights shows that $X$ has four fixed points:
\[
p_{12} := \PP(e_6 \wedge e_4 \wedge e_2),\;
p_{10} := \PP(e_6 \wedge e_4 \wedge e_0),\; p_{-10} := \PP(e_0 \wedge e_{-4} \wedge e_{-6}),\;
p_{-12} := \PP(e_{-2} \wedge e_{-4} \wedge e_{-6})
\]
under the Pl\"ucker embedding into $\PP(\wedge^3 U_6)$ (\cite[Section 5.2]{Don08}). Let us describe the $\CC^*$-invariant local coordinate charts of $X$ around these fixed points. Let $V_m^\circ$ denote the standard Schubert cell in Grassmannian $\Gr(3,U_6)$, and set
\[
V_m := X \cap V_m^\circ.
\]
The standard affine chart $V_{12}^\circ$ of $\Gr(3,U_6)$ around the fixed point $p_{12}$ is parametrized by
\[
\begin{aligned}
P &=\mathrm{span}\{e_6 + a_1 e_0 + a_2 e_{-2} + a_3 e_{-4} + a_4 e_{-6}, e_4 + a_5 e_0 + a_6 e_{-2} + a_7 e_{-4} + a_8 e_{-6}, \\
&e_2 + a_9 e_0 + a_{10} e_{-2} + a_{11} e_{-4} + a_{12} e_{-6}\}.
\end{aligned}
\]
The variety $X$ is defined by the quadratic condition $\eta (\wedge^2 P)=0$, where $\eta$ is the net of alternating forms on $U_6$ (\cite[Remark 3.12]{CKK25}).
Hence the affine chart $V_{12}$ is cut out by the following nine equations:
\begin{equation}\label{eq:u12}
\begin{aligned}
\langle 5a_2 + 3a_7,
a_3 + 9a_8, 
2a_2a_5 - 2a_1a_6 + \tfrac{1}{5}a_4&, 
10a_1 - a_{11},5a_2 - 9a_{12},\\
6a_2a_9 - 6a_1a_{10} - a_3,
a_5 + a_{10}, 
5a_6 + a_{11},
&2a_6a_9 - 2a_5a_{10} - \tfrac{1}{3}a_7 - \tfrac{1}{5}a_{12}\rangle.
\end{aligned}
\end{equation}
From these relations in \eqref{eq:u12}, one may choose $a_1, a_5, a_9$ as affine coordinates on $V_{12}$. Also by symmetry of weights up to sign, we may take $d_1, d_5, d_9$ as coordinates on $V_{-12}$. Note that both of these two charts $V_{12}$ and $V_{-12}$ are isomorphic to the affine space $\CC^3$ (cf. \cite[Section 5]{Fur90}). Similarly around the fixed point $p_{10}$, the affine chart $V_{10}^{\circ}$ of $\Gr(3, U_6)$ is parameterized by
\[
\begin{split}
Q&=\mathrm{span}\{ e_6+b_1e_2+b_2e_{-2}+b_3e_{-4}+b_4e_{-6}, e_4+b_5e_2+b_6e_{-2}+b_7e_{-4}+b_8e_{-6},\\
&b_9e_2+e_0+b_{10}e_{-2}+b_{11}e_{-4}+b_{12}e_{-6}\}
\end{split}
\]
and thus $V_{10}$ is defined as the nine equations:
\begin{equation}\label{bicoo}
\begin{aligned}
\langle -b_2-\frac{3}{5}b_7,b_2b_5-b_1b_6-\frac{1}{5}b_3-\frac{9}{5}b_8, -b_3b_5+b_1b_7+\frac{3}{5}b_4,-6b_1-\frac{3}{5}b_{11}, \;&b_2b_9-b_1b_{10}-\frac{9}{5}b_{12},\\
-b_3b_9+b_1b_{11}+6b_2, -6b_5+b_{10}, b_6b_9-b_5b_{10}+\frac{1}{5}b_{11}&, -b_7b_9+b_5b_{11}+6b_6-\frac{3}{5}b_{12}\rangle,
\end{aligned}
\end{equation}
which follows from the quadratic relation $\eta (\wedge^2 Q)=0$.
From the relations in \eqref{bicoo}, we know that the chart $V_{10}$ around the fixed point $p_{10}$ may have coordinates $b_5, b_8, b_9$. By the symmetry of the weights up to sign, we may take the coordinates of $V_{-10}$ as $c_5, c_8, c_9$.  Note that the open chart $V_{10}$ (resp. $V_{-10}$) is isomorphic to $V_{10}\cong \CC_{(b_5, b_8, b_9)}^3\setminus \{3b_5b_9-2=0\}$ (resp. $V_{-10}\cong \CC_{(c_5, c_8, c_9)}^3\setminus \{3c_5c_9-2=0\}$). In summary, the weights of the coordinates ${a_i}$, ${b_i}$, ${c_i}$, and ${d_i}$ are listed in Table \ref{tab:fixedx}.
Let us denote the \emph{weights} of the coordinate $-$ of $V_i$ by $\mathrm{wt}(-)$.
\begin{table}[htbp]
\centering
\begin{tabular}{|c|c|c|c|}
\hline
The local charts $V_{i}$ around the fixed points & Coordinates of $V_{i}$& Weights of coordinates \\ \hline\hline
$p_{12}\in V_{12}$ & $\{a_1, a_5, a_9\}$ & $\mathrm{wt}(a_1)=-6, \mathrm{wt}(a_5)=-4, \mathrm{wt}(a_9)=-2$ \\ \hline
$p_{10}\in V_{10}$ & $\{b_5, b_8, b_9\}$ & $\mathrm{wt}(b_5)=-2, \mathrm{wt}(b_8)=-10, \mathrm{wt}(b_9)=2$ \\ \hline
$p_{-10}\in V_{-10}$ & $\{c_5, c_8, c_9\}$ & $\mathrm{wt}(c_5)=2, \mathrm{wt}(c_8)=10, \mathrm{wt}(c_9)=-2$ \\ \hline
$p_{-12}\in V_{-12}$ & $\{d_1, d_5, d_9\}$ & $\mathrm{wt}(d_1)=6, \mathrm{wt}(d_5)=4, \mathrm{wt}(d_9)=2$ \\ \hline
\end{tabular}
\caption{Local coordinates of $X$ and its weights}
\label{tab:fixedx}
\end{table}
\begin{remark}
Let $\bG=\Gr(3,U_6)$ be the Grassmannian with the universal subbundle $\cS$. 
From the tangent bundle exact sequence
\begin{equation}\label{eq:tang}
\ses{T_{X}}{T_{\bG}|_{X}}{N_{X/\bG}\cong (\wedge^2\cS^*)^{\oplus3}},
\end{equation}
one can prove that the $\CC^*$-weights of the tangent space $T_{p_i}X$ at a fixed point $p_i$, 
$i\in\{\pm12,\pm10\}$ coincide with those obtained from the local chart $V_i$ of $X$ at the origin. For instance, the point $p_{10}$ corresponds to the subspace 
$U_3=\langle e_6,e_4,e_0\rangle\subset U_6$. 
The tangent space of $\bG$ at the point $p_{10}$ is identified with
\[
T_{\bG,p_{10}}
\cong \Hom(U_3,U_6/U_3)
\cong U_3^*\otimes (U_6/U_3).
\]
Moreover, the normal space of $X$ in $\bG$ at the point $p_{10}$ is given by
\[
N_{X/\bG,p_{10}}
\cong (\wedge^2 U_3^*)\otimes N,
\]
where the $\CC^*$-weights of $N$ are $\pm2$ and $0$ 
(\cite[Proposition 3.11]{CKK25}). Combining these descriptions with the exact sequence~\eqref{eq:tang}, 
we deduce that the $\CC^*$-weights of the tangent space $T_{X,p_{10}}$ 
are $\pm2, 10$. These agree with the weights computed from the local coordinates 
on $V_{10}$ at the origin, namely
\[
\mathrm{wt}\!\left(\frac{\partial}{\partial b_5}\right)=2,\,
\mathrm{wt}\!\left(\frac{\partial}{\partial b_8}\right)=10,\,
\mathrm{wt}\!\left(\frac{\partial}{\partial b_9}\right)=-2.\,
\]

\end{remark}
\subsection{Torus fixed curves and its local equations}
In \cite{CKK25}, the authors classify all $\mathbb{C}^*$-fixed curves of low degree using Kuznetsov's flop diagram (\cite[Theorem 1]{Kuz97}).
\begin{lemma}[\protect{\cite[Proposition 4.9 and Remark 4.10]{CKK25}}]\label{fixra}
The $\CC^*$-fixed, irreducible rational curves of degree $d\leq4$ in $X$ consist of
\begin{itemize}
\item the fixed lines $L_{\pm2}: \{e_{\pm6} \wedge e_{\pm4}\wedge (\alpha e_{\pm2}+\beta e_0)\}$,
\item the smooth conic $Q: \{e_0\wedge(\alpha e_{6}-9\beta e_{-4})\wedge (\alpha e_4+\beta e_{-6})\}$, and
\item the smooth quartic curve $C_4: \{(4\alpha^2 e_6+ \alpha \beta e_0-5\beta^2 e_{-6})\wedge (2\alpha e_4-\beta e_{-2})\wedge(2\alpha e_{2}+5  \beta e_{-4})\}$
\end{itemize}
for $[\alpha:\beta]\in \PP^1$. Furthermore, the local defining ideals of these curves are listed in TABLE \ref{ta:fixedcurve}.
\end{lemma}

\renewcommand{\arraystretch}{1.35}
\begin{table}[ht]
\centering
\begin{tabular}{|c|c|c|}
\hline
Fixed curves $C$ & Local charts $V_i$ of $X$ & Defining ideals of $C$ in $V_i$ \\
\hline
\multirow{2}{*}{$L_2$}
& $V_{12}$ &  $\langle a_1, a_5\rangle$
\\ \cline{2-3}
& $V_{10}$ &  $\langle b_5, b_8\rangle$
\\ \hline
\multirow{2}{*}{$L_{-2}$}
& $V_{-12}$ & $\langle d_1, d_5\rangle$
\\ \cline{2-3}
& $V_{-10}$ & $\langle c_5, c_8\rangle$
\\ \hline
\multirow{2}{*}{$Q$}
& $V_{10}$ & $\langle b_5, b_9\rangle$
\\ \cline{2-3}
& $V_{-10}$ & $\langle c_5, c_9\rangle$
\\ \hline
\multirow{2}{*}{$C_4$}
& $V_{12}$ & $\langle a_5, a_9\rangle$
\\ \cline{2-3}
& $V_{-12}$ & $\langle d_5, d_9\rangle$
\\ \hline
\end{tabular}
\caption{Fixed curves and their defining equations}
\label{ta:fixedcurve}
\end{table}
From now on, we denote by $L=L_2$ the first line in TABLE~\ref{ta:fixedcurve}.
\section{Computation of DT-invariants}
In this section, we will compute the integration through the localization formulas when $Y$ is a local CY $4$-fold over $X$.
\subsection{Computation of the integration via localization formulas}
Let $X$ be a smooth Fano $3$-fold. Let $Y = \mathrm{Tot}(K_X)$ be the total space of the canonical line bundle $K_X$ with the canonical projection map $\pi : Y \to X$. In this case, by \cite[Theorem 6.5]{CL14}, \[\bM_{\beta}(X)\cong \bM_{\pi^*\beta}(Y), F\mapsto i_*F\] for the zero section $i:X\hookrightarrow Y=\mathrm{Tot}(K_X)$. Also, by Grothendieck-Riemann-Roch theorem, the insertions become
\[
\tau_{i}(\gamma)=\pi_{\bM*}\left(\pi_X^*\gamma\cup \{\mathrm{ch}(\cF)\cdot \text{td}(K_X)^{-1}\}_{i+2}\right)
\]
and thus the integration in \eqref{intinse} is defined against on the virtual class $[\bM_{\beta}(X)]^{\mathrm{vir}}$ (\cite[Corollary 3.39]{Tho00}). In this setting, we formulate a localization formula for integrals with insertions.
\begin{proposition}\label{lcfinse}
Let us assume that the fixed locus $\bM^T:=\bM_{\beta}(X)^T$ is isolated. Then for $\gamma\in \rH_T^{4-2i}(X,\ZZ)$, the integration of the insertion is
\begin{equation}\label{lccom}
\langle \tau_i(\gamma)\rangle_{d}=\sum_{[F]\in \bM^T}\left(\sum_{p\in \mathrm{supp}(F)^T} \frac{\gamma|_p\cdot\{\mathrm{ch}^T(F|_p)\cdot \mathrm{td}^T(K_{X}|_p)^{-1}\}_{i+2}}{e^T(T_{X}|_p)}\right) e^T\left(\chi(F, F)-1\right).
\end{equation}
\end{proposition}
\begin{proof}
The equivariant-virtual tangent space at $[F]$ of $\bM:=\bM_{\beta}(X)$ is
\[
T_{[F]}^{\text{virt}}\bM=[\Ext^1(F,F)]-[\Ext^2(F,F)]=-\chi(F, F)+1
\]
since $\rH^1(\cO_X)=\Ext^3(F,F)=0$ by stability of the stable sheaf $F$ (\cite[Theorem 6.5]{CL14} and \cite[Corollary 3.39]{Tho00}). Therefore, by applying the virtual localization formula (Theorem \ref{virloc}), the integration of the insertion becomes
\begin{equation}\label{ab1}
\begin{split}
\langle \tau_i(\gamma)\rangle_{d}=\int_{[\bM]^{\mathrm{vir}}} \tau_i(\gamma)=\sum_{[F]\in \bM^T}\frac{\tau_i(\gamma)|_{[F]} }{e^T(T_{[F]}^{\text{vir}}\bM)}=\sum_{[F]\in \bM^T}\tau_i(\gamma)|_{[F]} \cdot e^T\left(\chi(F, F)-1\right).
\end{split}
\end{equation}
By applying localization formula (Remark \ref{smvir}) on $X$,
we have
\[
\tau_{i}(\gamma)|_{[F]}=\int_{X} \gamma \cdot \{\mathrm{ch}(F)\cdot \mathrm{td}(K_{X})^{-1}\}_{i+2}
=\sum_{p\in \text{supp}(F)^T} \frac{\gamma\cdot\{\mathrm{ch}^T(F)\cdot \mathrm{td}^T(K_{X})^{-1}\}_{i+2}|_p}{e^T(T_{X}|_p)}.
\]
Substituting this into \eqref{ab1}, we obtain the result.
\end{proof}
For the computation of invariants in \eqref{lccom} of Proposition \ref{lcfinse}, we consider the following. Let us denote the $T$-equivariant first Chern class of the trivial line bundle $\CC\otimes t$ by $c_1^T(\CC\otimes t)=\lambda$.
Suppose that the fiber of the hyperplane bundle $\cO_{X}(1)$ over the fixed point $p$ (in terms of the Plücker coordinates) is $\cO_{X}(1)|_{p}= \CC\otimes t^{m}, \; m\in \ZZ$. Then from the result of Section \ref{subge},
\begin{equation}\label{clssx}
h_1|_{p}=m\lambda,\; h_2|_{p}=\frac{1}{22}(m\lambda)^2,\; \mathrm{and}\; \text{td}^T(K_{X}|_p)^{-1}=\frac{1-e^{m\lambda}}{-m\lambda}.
\end{equation}
On the other hand, the equivariant Chern character of a stable sheaf $F$ supported on a smooth curve $i:C\subset X$ is given by (Theorem \ref{selfin})
\begin{equation}\label{cheq1}
\mathrm{ch}^T((i_*F)|_{p})=\mathrm{ch}^T\left([F]_p \cdot \lambda_{-1}(N_{C/X, p}^*)\right).
\end{equation}
After all, for the computation of the local form \eqref{lccom}, we need to know 
\begin{itemize}
\item the self-Euler pairing of the structure sheaf of fixed curves (cf. Theorem \ref{eqeuler}) and
\item the equivariant Chern character in \eqref{cheq1}.
\end{itemize}
On the other hand, one can easily check that the moduli space $\bM_4(X)$ 
is isomorphic to the Hilbert scheme $\bH_4(X)$ of curves $C$ with $\chi(\cO_C(m))=4m+1$, which parametrizes 
Cohen-Macaulay (CM) curves (\cite[Proposition 4.13]{CKK25}). Hence the $\CC^*$-fixed stable sheaves on $X$ 
are precisely the structure sheaves of $\CC^*$-fixed CM curves (\cite[Lemma 3.17]{CC11}).

\subsection{Equivariant Euler characteristics of fixed curves}\label{eq:ana}
In this subsection, we carry out the computation of the equivariant Euler characteristics of the fixed curves in $X$, which is one of the main technical ingredients of this paper. Recall that we regard the weights of $a_i$, $b_i$, $c_i$, $d_i$ (TABLE \ref{tab:fixedx}) as generators of the coordinate ring of the local chart $V_i$ of $X$ and thus they are the weights of cotangent (or conormal) sheaves of varieties. We begin by considering the case of irreducible $\CC^*$-fixed curves of degree $\leq 4$.
\begin{lemma}\label{d1l}
The equivariant Euler characteristic of the structure sheaf $\cO_{C}$ of the fixed curve $C\in \{L_{\pm2}, Q, C_4\}$ in TABLE \ref{ta:fixedcurve} is listed in TABLE \ref{tab:rateulr}.
\end{lemma}
\begin{table}[htbp]
\centering
\begin{tabular}{|c|l|M{6.5cm}|}
\hline
Rational curve $C$ & $\CC^*$-representation of $\chi(\cO_{C}, \cO_{C})$\\
\hline
\hline
$L_{2}$ & $1 - (t^{2} + t^{4}) + t^{8}$ \\
\hline
$L_{-2}$ & $1 - (t^{-2} + t^{-4}) + t^{-8}$ \\
\hline
$Q$ & $1- (t^{-2}+t^{2})+t^0$ \\
\hline
$C_{4}$ & $1-(t^{-4}+t^{-2}+t^{2}+t^{4})+(t^{-6}+t^0+t^{6})$ \\
\hline
\end{tabular}
\caption{Equivariant Euler pairing of torus fixed curves}
\label{tab:rateulr}
\end{table}
In Table \ref{tab:rateulr} and in similar situations below, parentheses indicate the deformation part (resp.\ obstruction part) $\Ext^1(\cO_C, \cO_C)$ (resp.\ $\Ext^2(\cO_C, \cO_C)$).
\begin{proof}
Since the fixed curves (TABLE \ref{ta:fixedcurve}) are locally defined by coordinate axes, the computation is essentially the same. Hence, we present the details for the first case, namely $L := L_2$. From Corollary \ref{smcurvechi}, Corollary \ref{lflfconc}, and $\chi(\cO_{L})=1$, we have
\[
\chi(\mathcal{O}_L, \mathcal{O}_L)
= 1-\left(\frac{N_{L/X, p_{12}}}{1 - T_{p_{12}}^*L}
+ \frac{N_{L/X, p_{10}}}{1 - T_{p_{10}}^*L}\right)+\left(\frac{\wedge^2N_{L/X,p_{12}}}{1 - T_{p_{12}}^*L}
+ \frac{\wedge^2N_{L/X,p_{10}}}{1 - T_{p_{10}}^*L}\right).
\]
In the local chart $V_{12}$ (resp. $V_{10}$) with coordinates $\{a_1,a_5,a_9\}$ (resp. $\{b_5,b_8,b_9\}$) (TABLE \ref{tab:fixedx}), the line $L$ is defined by $a_1=a_5=0$ (resp. $b_5=b_8=0$) (TABLE \ref{ta:fixedcurve}). By reading the weights of the normal and cotangent space of the line $L$ at the fixed points $p_{12}$ and $p_{10}$, one can see that
\[
\chi(\mathcal{O}_L, \mathcal{O}_L)=1- \left(\frac{t^{6}+t^{4}}{1-t^{-2}}+\frac{t^{2}+t^{10}}{1-t^{2}}\right)+\left(\frac{t^{10}}{1-t^{-2}}+\frac{t^{12}}{1-t^{2}}\right)=1-t^{2}-t^{4}+t^{8}.
\]
\end{proof}
\begin{remark}
Since the reduced scheme $\bM_1(X)_{\mathrm{red}}$ of the moduli space $\bM_1(X)$ is a smooth conic ($\cong\mathbf P^1$) and the action of $\mathrm{PGL}_2(\CC)=\mathrm{Aut}(\PP^1)$ is transitive (cf. \cite[Proposition 5.4.4]{KPS18}), the
local deformation spaces at all closed points are isomorphic to each other. As
$\dim \Ext^1(\mathcal O_L,\mathcal O_L)\geq2$ at the fixed point $[L]$ (Lemma \ref{d1l}), the same thing holds at every point, and thus the space $\bM_1(X)$ is everywhere non-reduced.
\end{remark}
Since higher-multiplicity curves supported a fixed line are not necessarily locally complete intersections, a study of their CM filtrations is required for the subsequent analysis. Note that the $\CC^*$-action on the line bundle $\cO_L(1)$ on the line $L$ is induced by the $\CC^*$-action on the hyperplane bundle $\cO_X(1)$. In particular, the action on the fiber $\cO_L(-1)|_{p_{12}}$ (resp. $\cO_L(-1)|_{p_{10}}$) has weight $-12$ (resp. $-10$).
\begin{proposition}\label{multfil}
There exists a unique multiplicity four curve $D_4$ in $X$ supported on the fixed line $\mathrm{red}(D_4)=L$ and $p_a(D_4)=0$. Furthermore, there exist CM-subcurves of $D_4$:
\begin{equation}\label{fil}
L:=D_1\subset D_2\subset D_3\subset D_4, \;\mathrm{mult}(D_k)=k
\end{equation}
such that
\begin{equation}\label{kclsd}
[\cO_{D_{k+1}}]=[\cO_{D_{k}}]+t^{10-2k}\cdot[\cO_L(-1)], \; 1\leq k \leq 3.
\end{equation}
\end{proposition}
\begin{proof}
In the proof of Proposition 4.1, 4.14, and 4.18 of \cite{CKK25}, the authors found the CM subcurves of $D_4$ with the CM filtration \eqref{fil}. In fact, let $\cI_{D_k}$ be the ideal of the curve $D_k$ in the standard Schubert chart $V_m^{\circ}\subset \Gr(3, U_6)$. Then $D_k$ is defined as
\[\begin{split}
\cI_{D_1}=&\langle a_1, a_2, a_3, a_4, a_5, a_6, a_7, a_8, a_{10}, a_{11},a_{12}\rangle,\\
\cI_{D_2}=&\langle a_1, a_2, a_3, a_4, a_6, a_7, a_8, a_{11}, a_{12}, 6a_5+a_{10}, a_{10}^2\rangle,\\
\cI_{D_3}=&\langle a_2, a_3, a_4, a_{7}, a_8, a_{12}, 5a_6+a_{11}, 6a_5+a_{10},  10a_1-a_{11}, a_{11}^2, a_{10}a_{11}, 5a_{10}^2-6a_9a_{11}\rangle,  \,\mathrm{and}\\
\cI_{D_4}=&\langle a_3, a_4, a_8, a_7+3a_{12}, 5a_6+a_{11}, 6a_5+a_{10},  5a_2-9a_{12}, 10a_1-a_{11}, a_{10}a_{12},a_{11}^2,\\
& a_{10}a_{11}-18a_9a_{12}, 5a_{10}^2-6a_9a_{11}+12a_{12} \rangle.\\
\end{split}
\]
Let $\cK_{k+1}$ be the kernel of the canonical restriction map: $\cO_{D_{k+1}}\twoheadrightarrow \cO_{D_{k}}$. Since $\cK_{k+1}\cong \cO_L(-1)$ and $\cK_{k+1}(V_{12})=\cI_{D_{k}}(V_{12})/\cI_{D_{k+1}}(V_{12})$ from its definition, let us pick up a generator of $\cK_{k+1}(V_{12})$ as an $\cO_L(U_{12})=\CC[a_9]$-module. Using Macaulay2 (\cite{M2}), we readily obtain that
\[
\cK_2(V_{12})=\langle \overline{a_5}\rangle,\;\; \cK_3(V_{12})=\langle \overline{a_1}\rangle, \;\mathrm{and}\;\cK_4(V_{12})=\langle \overline{a_{12}}\rangle.
\]
Since the weights of the generators are $\text{wt}(a_5)=-4$, $\text{wt}(a_1)=-6$, $\text{wt}(a_{12})=-8$, we finished the proof of the claim.
\end{proof}

\begin{corollary}\label{muchar}
Under the notation of Proposition~\ref{multfil}, we obtain the equivariant Euler characteristics of the curves $D_k$, as listed in Table~\ref{tab:multiple-curves-chi}.
\end{corollary}
\begin{table}[htbp]
\centering
\begin{tabular}{|c|l|M{6.5cm}|}
\hline
Curve $D_k$ & $T$-representation of $\chi(\cO_{D_k}, \cO_{D_k})$\\
\hline
\hline
$D_{1}$ & $1 - (t^{2} + t^{4}) + t^{8}$ \\
\hline
\hline
$D_{2}$ & $1 - (t^{2} + t^{4}) + t^{14}$ \\
\hline
\hline
$D_{3}$ & $1 - (t^{2} + t^{4} + t^{6}) + t^{14} + t^{16}$ \\
\hline
\hline
$D_{4}$ & $1 - (t^{2} + 2t^{4} + t^{6}) + t^{14} + t^{16} + t^{18}$ \\
\hline
\end{tabular}
\caption{Equivariant Euler pairing of multiple curves}
\label{tab:multiple-curves-chi}
\end{table}
\begin{proof}
By the additivity of the equivariant Euler pairing, the computation is inductive. The computation for the line $D_1=L$ was done in Lemma \ref{d1l}. Since other case is similar to each other, so we only present the computation for the curve $D_4$. From the equality $$[\cO_{D_{4}}]=[\cO_{D_{3}}]+t^{4}\cdot[\cO_L(-1)]$$ in \eqref{kclsd}, we know
\begin{equation}\label{chi4}
\chi(\cO_{D_4}, \cO_{D_4})=\chi(\cO_{D_3}, \cO_{D_3})+\chi(\cO_{L}, \cO_{L})+t^{-4}\chi(\cO_L(-1),  \cO_{D_3})+t^{4}\chi( \cO_{D_3}, \cO_L(-1)).
\end{equation}
By Theorem \ref{eqeuler}, Theorem \ref{selfin}, and Proposition~\ref{multfil}, the third term of \eqref{chi4} of the right side becomes
{\fontsize{8.5}{10}\selectfont\[
\begin{split}
&\chi(\cO_L(-1),  \cO_{D_3})=\frac{i_{p_{12}}^*([\cO_L(-1)])^\vee \cdot i_{p_{12}}^*([\cO_{D_3}])}
{\lambda_{-1}(T_{p_{12}}^*X)}+\frac{i_{p_{10}}^*([\cO_L(-1)])^\vee \cdot i_{p_{10}}^*([\cO_{D_3}])}
{\lambda_{-1}(T_{p_{10}}^*X)}.\\
&=\frac{i_{p_{12}}^*([\cO_L(-1)])^\vee \cdot i_{p_{12}}^*([\cO_L]+t^{8}[\cO_L(-1)]+t^{6}[\cO_L(-1)])}
{\lambda_{-1}(T_{p_{12}}^*X)}+\frac{i_{p_{10}}^*([\cO_L(-1)])^\vee \cdot i_{p_{10}}^*([\cO_L]+t^{8}[\cO_L(-1)]+t^{6}[\cO_L(-1)])}
{\lambda_{-1}(T_{p_{10}}^*X)}.\\
&=\frac{t^{12}(1-t^{6})(1-t^{4})\cdot(1+t^{8}\cdot t^{-12}+t^{6}\cdot t^{-12})(1-t^{-6})(1-t^{-4})}{(1-t^{-6})(1-t^{-4})(1-t^{-2})}+\frac{t^{10}(1-t^{2})(1-t^{10})\cdot(1+t^{8}\cdot t^{-10}+t^{6}\cdot t^{-10})(1-t^{-2})(1-t^{-10})}{(1-t^{-2})(1-t^{-10})(1-t^{2})}\\
&=t^{22}-t^{12}-t^{10}+t^{6}.
\end{split}
\]}
A similar computation yields that the last term in \eqref{chi4} is
\[
\chi(\cO_{D_3}, \cO_L(-1))=t^{2}-t^{-4}.
\]
By plugging these ones into \eqref{chi4} and using lower degree cases ($D_1$ and $D_3$) in TABLE \ref{tab:multiple-curves-chi}, we get the result.
\end{proof}
\begin{remark}
In Corollary~\ref{muchar}, the Euler characteristic of multiple lines (denoted by $D_{-k}$) supported on the fixed line $L_{-2}$ is obtained by reversing the signs of the weights.
\end{remark}
For the case of reducible fixed curves, we prove the following results.
\begin{lemma}\label{qleq}
Let $L=L_2$ and $Q$ be the fixed line and conic respectively such that the intersection point is $Q\cap L=\{p_{10}\}$. Then 
\[
\chi(\cO_{Q}, \cO_{L}(-1))=-t^{-12}+t^{-10}\;\mathrm{and} \;\chi(\cO_{L}(-1), \cO_{Q})=-t^{20}+t^{22}.
\]
\end{lemma}
\begin{proof}
Since $Q\cap L=\{p_{10}\}$, we only need to consider the localization in the local chart $V_{10}$ of $X$. Note that $Q$ is defined by $b_5=b_9=0$ in $V_{10}$ (TABLE \ref{ta:fixedcurve}). From Theorem \ref{eqeuler}, Theorem \ref{self} and Proposition~\ref{multfil}, we get
\[
\begin{split}
\chi(\cO_{Q}, \cO_{L}(-1))&=t^{-10}\cdot\frac{\lambda_{-1}(N_{Q/X, p_{10}}^*)^\vee \cdot \lambda_{-1}(N_{L /X, p_{10}}^*)}{\lambda_{-1}(T_{p_{10}}^*X)}=t^{-10}\cdot\frac{(1-t^2)(1-t^{-2})\cdot (1-t^{-2})(1-t^{-10})}{(1-t^{-2})(1-t^{-10})(1-t^{2})}\\
&=t^{-10}-t^{-12}\; \mathrm{and}\\
\chi(\cO_{L}(-1), \cO_{Q})&=t^{10}\cdot\frac{\lambda_{-1}(N_{L /X, p_{10}}^*)^\vee \cdot \lambda_{-1}(N_{Q/X, p_{10}}^*)}{\lambda_{-1}(T_{p_{10}}^*X)}=t^{10}\cdot\frac{(1-t^{2})(1-t^{10})\cdot(1-t^2)(1-t^{-2})}{(1-t^{-2})(1-t^{-10})(1-t^{2})}\\
&=t^{22}-t^{20}.
\end{split}
\]
\end{proof}
\begin{remark}
In the statement of Lemma \ref{qleq}, the Euler characteristic for the fixed line $L_{-2}$ is given by reversing the sign of weights.
\end{remark}
\begin{proposition}\label{pr:l2q}
Let $C=L^2\cup Q$ be the union of the planar double line $L^2(=D_2)$ and the smooth conic $Q$ such that $Q\cap L=\{p_{10}\}$.
Then
\begin{equation}\label{ql2}
[\cO_C]=[\cO_Q]+(t^{8}+t^{12})[\cO_L(-1)].
\end{equation}
Furthermore, $\chi(\cO_C, \cO_C)=1-(t^{-2}+2t^{2}+t^{6})+(t^0+t^{10}+t^{14})$.
\end{proposition}
\begin{proof}
Note that since $L$ and $Q$ intersect properly, $L\cup Q$ is a locally complete intersection (See TABLE \ref{ta:fixedcurve}).
Let $\cM$ (resp. $\cN$) be the kernel of the restriction map $\cO_C \twoheadrightarrow \cO_{L\cup Q}$ (resp. $\cO_{L\cup Q}\twoheadrightarrow \cO_Q$). Then $[\cM]=t^m[\cO_L(-1)]$, $[\cN]=t^n[\cO_L(-1)]$ and thus $[\cO_C]=[\cO_Q]+[\cM]+[\cN]$. Let us determine the weights $m$ and $n$ by a local computation. The defining ideals of $L\cup Q$ and $C$ over $V_{10}^\circ$ are respectively (\cite[Proposition 4.2 and Proposition 4.18]{CKK25})
\[\begin{split}
I_Q&=\langle b_1,b_2,b_4,b_5,b_6,b_7,b_3+9b_8,b_9, b_{10}, b_{11}, b_{12}\rangle,\\
I_{L\cup Q}&= \langle b_1,b_2,b_4,b_5,b_6,b_7,b_{10},b_{11},b_{12}, b_3+9b_8, b_8 b_9\rangle, \;\mathrm{and}\\
I_{C}&=\langle b_{12},b_{11},b_7, b_6, 6b_5-b_{10},b_4,b_3+9b_8,b_2,b_1,b_{10}^2,b_8b_{10},b_8b_9\rangle.
\end{split}
\]
With the help of Macaulay2 program (\cite{M2}) again, \[\cM(V_{10})=I_{L\cup Q}/I_{C}=\langle \bar{b}_5\rangle\; \mathrm{and}\; \cN(V_{10})=I_{Q}/I_{L\cup Q}=\langle \bar{b}_9\rangle.\] Therefore, $m=8$ and $n=12$ since the weights are $\text{wt}(b_5)=-2$, $\text{wt}(b_9)=2$, and $\text{wt}(\cO_L(-1)|_{p_{10}})=-10$ (TABLE \ref{tab:fixedx} and Proposition \ref{multfil}). After all, we proved the equality \eqref{ql2}.
On the other hand, from the linearity of Euler characteristic, we see that
\[\begin{split}
\chi(\cO_{C}, \cO_{C})&=\chi(\cO_Q, \cO_Q)+(2+t^4+t^{-4})\chi(\cO_L, \cO_L)\\
&+ (t^{12}+t^8)\chi(\cO_Q, \cO_L (-1)) +(t^{-12}+t^{-8})\chi(\cO_L(-1), \cO_Q)\\
&=2 -t^{-2}- 2t^{2} - t^{6}+t^{10} + t^{14},
\end{split}
\]
where the second equality holds from TABLE \ref{tab:rateulr} and Lemma \ref{qleq}.
\end{proof}
\begin{remark}\label{re:def}
From Lemma \ref{d1l}, Lemma \ref{qleq} and the proof of Proposition \ref{pr:l2q}, one can easily see that
\[\chi(\cO_{L\cup Q},\cO_{L\cup Q})=1-(t^{-2}+t^{2}+t^4)+(t^0+t^{10}).\]
\end{remark}
\begin{proposition}\label{qll}
Let $C= L_2 \cup L_{-2}\cup Q$ such that all of subcurves $Q$, $L_{2}$, and $L_{-2}$ are $\CC^*$-fixed ones. Then
\begin{equation}\label{eq12}
[\cO_C]=[\cO_Q]+t^{12}[\cO_{L_2}(-1)]+t^{-12}[\cO_{L_{-2}}(-1)].
\end{equation}
Furthermore, $\chi(\cO_C, \cO_C)=1-(t^{-4}+t^{-2}+t^{2}+t^{4})+(t^{-10}+t^0+t^{10})$.
\end{proposition}
\begin{proof}
By the symmetry of $C$ and the proof of Proposition \ref{pr:l2q}, the equality \eqref{eq12} holds. On the other hands, the Euler pairing becomes
\[\begin{split}
\chi(\cO_C, \cO_C)&=\chi(\cO_Q, \cO_Q)+\chi(\cO_{L_2}, \cO_{L_2})+\chi(\cO_{L_{-2}}, \cO_{L_{-2}})\\
&+t^{12}\left(\chi(\cO_Q, \cO_{L_{-2}}(-1))+\chi(\cO_{L_{-2}}(-1), \cO_Q)\right)+t^{-12}\left(\chi(\cO_Q, \cO_{L_2}(-1))+\chi(\cO_{L_2}(-1), \cO_Q)\right)\\
&=2-t^{-4}-t^{-2}-t^{2}-t^{4}+t^{10}+t^{-10}
\end{split}
\]
from Lemma \ref{d1l} and Lemma \ref{qleq}. Note that $\chi(\cO_{L_1}, \cO_{L_2})=0$ since $L_2\cap L_{-2}=\emptyset$.
\end{proof}
\begin{remark}
In Lemma \ref{d1l}, Corollary \ref{muchar}, Proposition \ref{pr:l2q}, and Proposition \ref{qll}, the weights arising from the deformation part coincide with those in Section 4.3 of \cite{CKK25}.
\end{remark}
\subsection{Integration of the primary and descendent insertions}
By using the results in the Section \ref{lcfinse}, we compute the integration of insertions which will be used for the proof of Conjecture \ref{conjor}.
\begin{proposition}\label{mainpro}
$\langle \tau_0(h_2)\rangle_{4}=168$, $\langle \tau_1(h_1)\rangle_{4}=-168$ and $\langle \tau_2(h_0:=1)\rangle_{4}=-28$.
\end{proposition}
\begin{proof}
The fixed locus of $\bM_4$ consists of the following six $T$-fixed stable sheaves (see \cite[Section 4.3]{CKK25}):
\[
\bM_4^T
=
\left\{
[\cO_C]
\;\middle|\;
C \in \{D_{\pm4},\, L_{\pm2}^2 \cup Q,\, L_2 \cup L_{-2} \cup Q,\, C_4\}
\right\}.
\]
Each curve in $\bM_4^T$, except for the multiple line $D_{\pm4}$, has a zero-weight part in the obstruction space, and hence its Euler class vanishes (Lemma \ref{d1l}, Proposition \ref{pr:l2q} and Proposition \ref{qll}). Therefore, it suffices to consider the contribution from $D_{\pm4}$. Since $D_4$ and $D_{-4}$ are symmetric, their contributions to the integration are the same. Hence for the computation of local form \eqref{lccom} of Proposition \ref{lcfinse}, we need the followings. Over the fixed point $p_{12}$ (resp. $p_{10}$), the restriction of the hyperplane bundle is $\cO_X(1)|_{p_{12}}=\CC\otimes t^{12}$ (resp. $\cO_X(1)|_{p_{10}}=\CC\otimes t^{10}$) and thus one can obtain the character $h_i|_{p_{j}}$ and the inverse of equivariant Todd classes \eqref{clssx}. Also, from \eqref{kclsd} and Theorem \ref{selfin}, we know the equivariant Chern character \eqref{cheq1} of $\cO_{D_4}$ along the fixed points. Finally, the weights of tangent space $T_X$ at each fixed point $p$ are given in TABLE \ref{tab:fixedx}. As summery, we obtain the list in TABLE \ref{tab:D4_localization}.
\renewcommand{\arraystretch}{1.1}
\begin{table}[htbp]
\begin{center}
\resizebox{\textwidth}{!}{%
\begin{tabular}{|c|l|M{2.2cm}|M{2.2cm}|M{6.2cm}|M{3.6cm}|}
\hline
\multirow{2}{*}{Fixed curve}
& $p_1$
& $h_1|_{p_1}$
& $h_2|_{p_1}$
& $\mathrm{ch}^T(F|_{p_1})$
& $e^T(T_{X}|_{p_1})$
\\
\cline{2-6}
& $p_2$
& $h_1|_{p_2}$
& $h_2|_{p_2}$
& $\mathrm{ch}^T(F|_{p_2})$
& $e^T(T_{X}|_{p_2})$
\\
\hline\hline
\multirow{2}{*}{$D_{4}$}
& $p_{12}$
& $12\lambda$
& $\frac1{22}(12\lambda)^2$
& $\displaystyle
\begin{aligned}
&(1+e^{-4\lambda}+e^{-6\lambda}+e^{-8\lambda})\\
&\cdot(1-e^{-6\lambda})(1-e^{-4\lambda})
\end{aligned}$
& $\displaystyle (4\lambda)\cdot(6\lambda)\cdot(2\lambda)$
\\
\cline{2-6}
& $p_{10}$
& $10\lambda$
& $\frac1{22}(10\lambda)^2$
& $\displaystyle
\begin{aligned}
&(1+e^{-2\lambda}+e^{-4\lambda}+e^{-6\lambda})\\
&\cdot(1-e^{-2\lambda})(1-e^{-10\lambda})
\end{aligned}$
& $\displaystyle (2\lambda)\cdot(10\lambda)\cdot(-2\lambda)$
\\
\hline
\end{tabular}%
}
\caption{Localization data at the fixed curve $D_4$.}
\label{tab:D4_localization}
\end{center}
\end{table}
One the other hand, the equivariant Euler class of $\chi(\cO_{D_4}, \cO_{D_4})-1$ (TABLE \ref{tab:multiple-curves-chi}) is
\[
e^T(\chi(\cO_{D_4}, \cO_{D_4})-1)=\frac{(14\lambda)(16\lambda)(18\lambda)}{(2\lambda)(4\lambda)^2(6\lambda)}.
\]
By plugging this one and the data of TABLE \ref{tab:D4_localization} into \eqref{lccom}, we finally obtain the invariants as follows.
\[
\begin{aligned}
\langle \tau_0(h_2)\rangle_{4}
&= 2\cdot\Bigg(
\frac{\frac1{22}(12\lambda)^2\cdot
\Bigl\{\mathrm{ch}^T(\cO_{D_4}|_{p_{12}})\cdot\frac{1-e^{12\lambda}}{-12\lambda}\Bigr\}_{2}}
{(4\lambda)(6\lambda)(2\lambda)}
+
\frac{\frac1{22}(10\lambda)^2\cdot
\Bigl\{\mathrm{ch}^T(\cO_{D_4}|_{p_{10}})\cdot\frac{1-e^{10\lambda}}{-10\lambda}\Bigr\}_{2}}
{(2\lambda)(10\lambda)(-2\lambda)}
\Bigg)\cdot \frac{(14\lambda)(16\lambda)(18\lambda)}{(2\lambda)(4\lambda)^2(6\lambda)}
\\
&=168,
\\
\langle \tau_1(h_1)\rangle_{4}
&= 2\cdot\Bigg(
\frac{(12\lambda)\cdot
\Bigl\{\mathrm{ch}^T(\cO_{D_4}|_{p_{12}})\cdot\frac{1-e^{12\lambda}}{-12\lambda}\Bigr\}_{3}}
{(4\lambda)(6\lambda)(2\lambda)}
+
\frac{(10\lambda)\cdot
\Bigl\{\mathrm{ch}^T(\cO_{D_4}|_{p_{10}})\cdot\frac{1-e^{10\lambda}}{-10\lambda}\Bigr\}_{3}}
{(2\lambda)(10\lambda)(-2\lambda)}
\Bigg)\cdot \frac{(14\lambda)(16\lambda)(18\lambda)}{(2\lambda)(4\lambda)^2(6\lambda)}
\\
&=-168,
\\
\langle \tau_2(1)\rangle_{4}
&= 2\cdot\Bigg(
\frac{
\Bigl\{\mathrm{ch}^T(\cO_{D_4}|_{p_{12}})\cdot\frac{1-e^{12\lambda}}{-12\lambda}\Bigr\}_{4}}
{(4\lambda)(6\lambda)(2\lambda)}
+
\frac{\Bigl\{\mathrm{ch}^T(\cO_{D_4}|_{p_{10}})\cdot\frac{1-e^{10\lambda}}{-10\lambda}\Bigr\}_{4}}
{(2\lambda)(10\lambda)(-2\lambda)}
\Bigg)\cdot \frac{(14\lambda)(16\lambda)(18\lambda)}{(2\lambda)(4\lambda)^2(6\lambda)}
\\
&=-28.
\end{aligned}
\]
\end{proof}
For $d \leq 3$, it follows from Lemma \ref{d1l} and Remark \ref{re:def} that the fixed curves $D_{\pm k}$, $1\leq k\leq 3$ make only contribution to the integrals of the insertions. Therefore, the following proposition is obtained by suitably adjusting the equivariant Chern character and the Euler class of virtual tangent space of $\bM_d$.
\begin{proposition}[cf. \protect{\cite[Proposition 3.17]{CLW24}}]\label{v22dt}
The invariants $\langle \tau_i(h_{2-i})\rangle_{d}$, ($1\leq d\leq 4$) are given by the numbers of the following table:
\begin{table}[htbp]
\begin{center}
\begin{tabular}{|l||M{2.2cm}|M{2.2cm}|M{2.2cm}|}
\hline
$d$ & $i=0$ & $i=1$ & $i=2$\\
\hline\hline
$1$ & $2$   & $22$   & $-\frac{1}{3}$\\
\hline
$2$ & $7$   & $28$   & $-\frac{7}{6}$\\
\hline
$3$ & $28$  & $28$   & $-\frac{14}{3}$\\
\hline
$4$ & $168$ & $-168$ & $-28$\\
\hline
\end{tabular}
\caption{The integrations $\langle \tau_i(h_{2-i})\rangle_{d}$}
\label{tab:invariants_d1to4}
\end{center}
\end{table}\end{proposition}
\begin{remark}
In TABLE \ref{tab:invariants_d1to4} of Proposition \ref{v22dt}, the ratio of invariants for $i=0$ and $2$ are constant: $\frac{\langle \tau_2(h_0)\rangle_{d}}{\langle \tau_0(h_2)\rangle_{d}}=-\frac{1}{6}$ for all $1\leq d\leq 4$. This verifies the prediction in \cite[Section~1.7]{CT21}.
\end{remark}
\subsection{Proof of Conjecture \ref{conjor}}\label{subproof}
Let us recall the conjecture between $\mathrm{DT}$ and GV invariants. Let us denote $n_{0,\beta}(\gamma)$ (resp. $n_{1,\beta}$) be the genus 0 (resp. genus 1) Gopakumar-Vafa invariants on a CY 4-fold $Y$ defined in \cite{KP08}.
The number $m_{\beta_1,\beta_2}$ is a virtual count of degree $\beta_1$ curves in $Y$ which meets with degree $\beta_2$ curves, called \emph{meeting invariants} (\cite[Section 0.3]{KP08}). In our setting, the meeting invariants $m_{\beta_1,\beta_2}\in \ZZ$ is determined by the following inductive rules. Recall that the cohomology degree of $h_i\in \rH^*(X)$ is $\deg(h_i)=2i$.
\begin{enumerate}[(a)]
\item \label{a:a} $m_{\beta_1\!,\,\beta_2}=m_{\beta_2,\,\beta_1}$. 
\item \label{b:b} Also, if either $\deg(\beta_1)\leq 0$ or $\deg(\beta_2)\leq 0$, then $m_{\beta_1,\beta_2}=0$.
\item \label{c:c} If $\beta_1\neq \beta_2$, then $m_{\beta_1,\beta_2}=-22n_{0,\beta_1}(h_2)n_{0,\beta_2}(h_2)+m_{\beta_1,\beta_2-\beta_1}+m_{\beta_1-\beta_2,\beta_2}.
$
\item \label{d:d} If $\beta_1= \beta_2=\beta$, then $
m_{\beta,\beta}=2n_{0,\beta}(h_2)-22 n_{0,\beta}(h_2) n_{0,\beta}(h_2)-\sum_{\beta_1+\beta_2=\beta}m_{\beta_1,\beta_2}.
$
\end{enumerate}
Here we use the relation of the invariants on $Y$ and those of $X$ as follow. Let $\overline{Y}=\mathbb{P}(K_X\oplus \mathcal{O}_X)$ be a projective compactification of $Y$ with a zero section $i: X\hookrightarrow \overline{Y}$. Let $S_1:=i_*h_1$ and $S_2:=\pi^*h_2$ be the basis of $\rH^4(\overline{Y}, \ZZ)$. Then
\[
n_{0,d}(S_1)=-22n_{0,d}(h_2), \;n_{0,d}(S_2)=n_{0,d}(h_2)
\]
by the self-intersection formula (For the detail computation, see \cite[Section 2.4]{CT21} and \cite[Section 4.1]{CLW24}). 
Under this setting, the equality \eqref{eq:mov1} of Conjecture \ref{conjor} becomes
\begin{align}\label{eqct}
 \langle \tau_1(h_1) \rangle_{\beta} & =\frac{11n_{0,\beta}(h_2)}{h_1\cdot \beta}-\sum_{\beta_1+\beta_2=\beta} \frac{(h_1\cdot \beta_1)(h_1\cdot \beta_2)}{4(h_1\cdot \beta)} m_{\beta_1,\beta_2}-\sum_{k\geq1, k|\beta} \frac{(h_1\cdot \beta)}{k}n_{1,\beta/k}
\end{align}
for $\beta=d[\mathrm{line}]$ under the some suitable choice of orientation on the moduli space $\bM_d(X)$. It is expected that the invariants $n_{1,d}$ come from a space of degree $d$ elliptic curves in $Y$ (cf. \cite[Section 3 and 5]{KP08}).By \cite[Corollary 3.9]{CFK23}, the space of elliptic curves of degree $d \leq 6$ is empty. Therefore, it is reasonable to expect that $n_{1,d}=0$.
\begin{theorem}\label{pfmeeting1}
Conjecture \ref{conjor} holds for $Y=\mathrm{Tot}(K_X)$ when we assume that $n_{1,d}=0$ for $1\leq d\leq 4$.
\end{theorem}
\begin{proof}
Since $n_{1,d}=0$ for $1 \le d \le 4$, the conjectural identity in \eqref{eqct} can be written as
\begin{equation}
\langle \tau_1(h_1) \rangle_d
=\frac{11n_{0,d}(h_2)}{d}
-\sum_{d_1+d_2=d}\frac{d_1 d_2}{4d}\, m_{d_1,d_2},\; 1 \le d \le 4,
\end{equation}
Under the suitable choice of the orientation of the virtual class $[M_d]^{\mathrm{vir}}$.
Following the orientation of the paper \cite{CLW24}, 
the sign of a virtual cycle of the moduli space $M_d$ is given by $(-1)^{d+1}$. 
Then for our case, we must prove the following equality:
\begin{equation}\label{fineq}
-\langle \tau_1(h_1) \rangle_4=-\frac{11\times \langle \tau_0(h_2) \rangle_4}{4}-\frac{1}{16}\times\left(6\times m_{1,3}+4\times m_{2,2}\right).
\end{equation}
Firstly, note that
\[
m_{1,1}=-84, \;
m_{1,2}=224, \;
n_{0,1}(h_2)=2, \;
n_{0,2}(h_2)=-7, \;
n_{0,3}(h_2)=28
\]
by \cite{CLW24} or a direct computation using the values in TABLE \ref{tab:invariants_d1to4}. Secondly, by the recursive formulas (\ref{a:a})--(\ref{d:d}) of meeting invariants, the meeting invariants $m_{1,3}$ and $m_{2,2}$ of the variety $X$ are given by
\begin{equation*}
\begin{split}
m_{1,3}&=-22 n_{0,1}(h_2) n_{0,3}(h_2) + m_{1,2}=-22 \times 2 \times 28 + 224=-1008.\\
m_{2,2}&=2 n_{0,2}(h_2)-22 n_{0,2}(h_2)^2-m_{1,1} =2 \times (-7)-22 \times (-7)^2+84=-1008.
\end{split}
\end{equation*}
Plugging these numbers and the results in TABLE \ref{tab:invariants_d1to4} into the both sides of equality \eqref{fineq}, we finish the proof of the claim.
\end{proof}
\bibliographystyle{alpha}
\newcommand{\etalchar}[1]{$^{#1}$}

\end{document}